\documentclass[12pt, reqno]{amsart}
\usepackage{amssymb,dsfont}
\usepackage{eucal}
\usepackage{amsmath}
\usepackage{amscd}
\usepackage[dvips]{color}
\usepackage{multicol}
\usepackage{graphicx}
\usepackage{color}
\usepackage{colordvi}
\usepackage{xspace}
\usepackage{txfonts}
\usepackage{lscape}
\usepackage{tikz} 
\numberwithin{equation}{section}
\usepackage[shortlabels]{enumitem}
\usepackage{ifpdf}
\ifpdf
\usepackage[colorlinks,final,backref=page,hyperindex]{hyperref}
\else
\usepackage[colorlinks,final,backref=page,hyperindex]{hyperref}
\fi
\usepackage{tikz}
\usepackage[active]{srcltx}

\makeatletter

\def\beq{\begin{equation}}
\def\eeq{\end{equation}}
\def\beqn{\begin{equation*}}
\def\eeqn{\end{equation*}}
\newtheorem{thm}{Theorem}[section]
\newtheorem{lem}[thm]{Lemma}
\newtheorem{prop}[thm]{Proposition}
\newtheorem{cor}[thm]{Corollary}
\newtheorem{Def}{Definition}[section]

\def\Fg{\mathfrak{g}}
\def\Fsl{\mathfrak{sl}}
\def\Fsp{\mathfrak{sp}}
\def\Fh{\mathfrak{h}}

\def\Fn{\mathfrak{n}}

\def\CC{\mathbb{C}}
\def\S{\mathcal{S}}

\def\Hom{\mathrm{Hom}}
\def\Der{\mathrm{Der}}
\def\IDer{\mathrm{IDer}}

\def\T{\mathrm{T}}
\def\wt{\mathrm{wt}}
\def\a{\alpha}

\def\g{\gamma}
\def\d{\delta}
\def\la{\lambda}
\def\La{\Lambda}
\def\o{\overline}

\def\<{\left<}
\def\>{\right>}
\def\({\left(}
\def\){\right)}
\def\[{\left[}
\def\]{\right]}
\def\<{\left<}
\def\>{\right>}
\counterwithin{equation}{section}
\title{Symmetric Biderivations on Complex Semisimple Lie Algebras}

\author{Shiyuan Liu}
\address{Department of Mathematics, Xinyang Normal University, Henan Xinyang, 46400 , China}
\email{liushiyuanxy@163.com}

\author{Dong Liu}
\address{Department of Mathematics, Huzhou
University, Zhejiang Huzhou, 313000, China}
\email{liudong@zjhu.edu.cn}

\author{Yueqiang Zhao}
\address{Department of Mathematics, Xinyang Normal University, Henan Xinyang, 46400 , China}
\email{yueqiangzhao@163.com }

\thanks{This work was supported by the National Natural Science Foundation of China (12071405, 12301040) and Nanhu Scholars Program of XYNU (No. 012111, 2021030)}

\begin{document}

\maketitle

\begin{abstract} 
This paper studies  biderivations on finite-dimensional complex semisimple Lie algebras to their finite-dimensional modules. More precisely, we prove that all such symmetric biderivations are trivial. As applications, we   determine all biderivations on
Takiff algebras and symplectic oscillator algebras, as well as 
all supersymmetric biderivations over some Lie superalgebras including all finite-dimensional classical simple Lie superalgebras.
\end{abstract}

\medskip


\medskip



\section{Introduction}

Derivations and generalized derivations are important in the study of structure of various algebras. Bre$\breve{\mathrm{s}}$ar et al. introduced the notion of biderivation of rings in \cite{B93}, and they showed that all biderivations of noncommutative prime rings are inner.  In \cite{W11}, the biderivations of Lie algebras was introduced and the authors proved that all skew-symmetric biderivations of finite-dimensional simple Lie algebras over an algebraically closed field of characteristic $0$ are inner. In recent years, many scholars gave their attentions to the study of biderivations of many Lie algebras using case by case discussions, see \cite{C16,Ch17,T18,T17,L18,T18(1)}. 
Certainly, some complicated calculations were used in many papers.

\medskip

The skew-symmetric biderivations of Lie algebras and Lie superalgebras have been sufficiently studied. In \cite{B18, TMC}, the authors proved that all (super-)skew-symmetric biderivations on any perfect and centerless Lie (super)algebras are inner (super-)biderivations. With this result, all skew-symmetric biderivations on many Lie (super)algebras can be determined (see \cite{B18} for details).

\medskip

It is well known that any biderivation can be decomposed into a sum of a skew-symmetric biderivation and a symmetric biderivation, and the later can determined commutative post-Lie algebra structures, which are connected with the homology of partition posets and Koszul operads \cite{V}.
However, there is no any sufficient tool to determine symmetric biderivations on some Lie algebras and Lie superalgebras. In \cite{Chen24}, Chen et al. use a conceptual method to determine all symmetric biderivations on finite-dimensional simple Lie algebras. In \cite{W11}, Wang et al. study  symmetric biderivations from the simple Lie algebra $\frak{sl}_2$  to its finite-dimensional irreducible modules by concrete calculations.
In \cite{D23}, Dilxat et al. determine all symmetric biderivations from the Virasoro algebra to its density modules. Based on it,  all symmetric biderivations of some Lie superalgeras related to the Virasoro algebra are easily to determined. 

\medskip

From \cite{Chen24, W11, D23},  we seen that it is important to determine all symmetric biderivations from a finite-dimensinoal semisimple complex Lie algebra to its arbitrary finite-dimensional module.  Inspired by  \cite{Chen24},   
we  proved that  all the symmetric biderivations from a finite-dimensinoal semisimple complex Lie algebra to its arbitrary finite-dimensional module are trivial in this paper (Theorem \ref{T3.1} below), by using some representation theory of semisimple Lie algebras. By this main result, we can easily determine all the symmetric biderivations of reductive Lie algebras, Takiff algebras, the symplectic oscillator algebras,  and the super-symmetric biderivations of certain Lie superalgebras.

\medskip

Throughout this paper,   all vector
spaces and algebras  are over the complex field $\mathbb C$.


\section{Biderivations}
We first recall the notion of derivations.
Let $\Fg$ be a Lie algebra and $V$ be a $\Fg$-module. A linear map $D:\Fg\longrightarrow V$ is called a derivation from $\Fg$ to $V$ if 
\beqn
D([x,y])=xD(y)-yD(x),~\forall x,y\in\Fg.
\eeqn
Note that a derivation from $\Fg$ to its adjoint representation is just an usual derivation of $\Fg$. 
For any $v\in V$, the linear map
\beqn
D_v:\Fg\longrightarrow V,~x\longmapsto xv
\eeqn
is a derivation from $\Fg$ to $V$, which is called  inner. Denote by $\Der(\Fg,V)$ the set of derivations from $\Fg$ to $V$ and $\IDer(\Fg,V)$ the set of inner derivations. Then $\Der(\Fg,V)$ is a linear space and  $\IDer(\Fg,V)$ is a subspace of $\Der(\Fg,V)$. The 1st cohomology group is defined as 
\beqn
H^1(\Fg,V)=\Der(\Fg,V)/\IDer(\Fg,V).
\eeqn
We have the following Whitehead's Lemma
\begin{thm}[{\cite[Proposition 6.1]{H96}}]\label{White}
Let $\Fg$ be a semisimple Lie algebra over a field of characteristic $0$, $V$ a finite dimensional $\Fg$-module. Then $H^1(\Fg,V)=0$. In particular, the derivations from $\Fg$ to $V$ are all inner.
\end{thm}

We now introduce the notion of biderivations.

\begin{Def}
Let $\Fg$ be a Lie algebra and $V$ be a $\Fg$-module. A bilinear map $\d:\Fg\times \Fg\longrightarrow V$ is called a biderivation from $\Fg$ to $V$ if 
\begin{align}
&\delta([x,y],z)=x\delta(y,z)-y\delta(x,z),\\
&\delta(x,[y,z])=y\delta(x,z)-z\delta(x,y)
\end{align}
hold for all $x,y,z\in\Fg$.
\end{Def}
Note that a biderivation from $\Fg$ to its adjoint representation is just a usual biderivation of $L$. It is elementary to check that a bilinear map $\d:\Fg\times\Fg\longrightarrow V$ is a biderivation if and only if the linear maps
\begin{align*}
& \d(\cdot, x):\Fg\longrightarrow V,~y\longmapsto \d(y,x),\\
&  \d(x,\cdot):\Fg\longrightarrow V,~y\longmapsto \d(x,y)
\end{align*}
are both derivations from $\Fg$ to $V$ for any $x\in\Fg$.
Assume that $\g:V\longrightarrow W$ is a $\Fg$-module homomorphism. Then the bilinear map $\g\d:\Fg\times\Fg\longrightarrow W$ is a biderivation from $V$ to $W$.

\medskip

A biderivation $\d$ from $\Fg$ to $V$ is said to be symmetric (resp. skew-symmetric) if $\d(x,y)=\d(y,x)$ (resp. $\d(x,y)=\d(y,x)$) for all $x,y\in\Fg$.  It is well-know that any biderivation can be decomposed into a sum of a symmetric biderivation and a skew-symmetric biderivation. For skew-symmetric biderivations, we have the following theorem

\begin{thm}[{\cite[Theorem 3.2]{B18}}]\label{T2.2}
Let $\Fg$ be a perfect Lie algebra and $V$ be a $\Fg$-module such that 
\beqn
Z_V(\Fg):=\{v\in V\mid xv=0~\text{for all}~x\in \Fg\}=\{0\}.
\eeqn
Then every skew-symmetric biderivation $\d:\Fg\times\Fg\longrightarrow V$ is of the form $\d(x,y)=\g([x,y])$ for some $\g\in\Hom_\Fg(\Fg,V)$.
\end{thm}
In contrast to the skew-symmetric biderivations, there are no sufficient tools to determine the symmetric biderivations. In this paper, we mainly study the symmetric biderivations from a complex semisimple Lie algebra to its finite-dimensional module. The $\Fsl_2(\CC)$ case has been studied in \cite{Wang22}. 

\begin{lem}[\cite{Wang22}]\label{L2.3}
Let $V$ be an irreducible $\Fsl_2(\CC)$-module. Then all the biderivations  from $\Fsl_2(\CC)$ to $V$ are trival.
\end{lem}

Now we show the following corollary which will be used in the next section.

\begin{cor}\label{C2.4}
All the biderivations  from $\Fsl_2(\CC)$ to its finite-dimensional module are trival.
\end{cor}
\begin{proof}
Let $V$ be a finite-dimensional $\Fsl_2$-module. By Weyl's theorem, $V$ is completely irruducible, i.e., V can be written as a direct sum of its irreducible submodules, say
\beqn
V=V_1\oplus V_2\oplus \cdots\oplus V_m.
\eeqn
For $1\leq i\leq m$, let $p_i:V\longrightarrow V_i$ be the canonical projection. Let $\d$ be a symmetric biderivation from $\Fsl_2$ to $V$. Then for each $1\le i\le m$, $p_i\d$ is a symmetric biderivation from $\Fsl_2$ to $V_i$. By the above lemma, $p_i\d=0$. It follows that the image of $\d$ is $0$. Hence, $\d=0$.
\end{proof}

\section{Main Result}

We first recall some basic facts about finite-dimensional semisimple complex Lie algebras and their finite dimensional modules.

\medskip
 
In the following, let $\Fg$ be a finite-dimensional complex semisimple Lie algebra, all the $\Fg$-modules are assumed to be   finite-dimensional. Let $\Fh$ be a  Cartan subalgebra of $\Fg$. Then $\Fg$ admits a Cartan decomposition 
\beqn
\Fg=\Fh\oplus\sum_{\a\in\Phi} \Fg_{\a},
\eeqn
where $\Phi\subset\Fh^*$ is the  root system. Let $\Phi^+$ (resp. $\Phi^-$) be the set of positive (resp. negative) roots with respect to a certain chosen base. Set
\beqn
\Fn^+=\sum_{\a\in\Phi^+}\Fg_\a,~\Fn^-=\sum_{\a\in\Phi^-}\Fg_\a.
\eeqn
For each $\a\in\Phi^+$, choose an nonzero element $e_{\a}\in \Fg_{\a}$. Then there exist  unique 
$f_\a\in\Phi^-$ and $h_\a\in\Fh$ such that 
\beqn
[h_\a,e_\a]=2e_\a,~[h_\a,f_\a]=-2f_\a,~[e_\a,f_\a]=h_\a.
\eeqn
Hence $S_\a=\CC e_\a\oplus\CC h_\a\oplus\CC f_\a$ is a subalgebra of $\Fg$, which is isomorphic to $\Fsl_2(\CC)$. It is well known that $\Fg$ is spanned by the elements $e_\a,f_\a,h_\a~\a\in\Phi^+$.

\medskip

Let $V$ be a finite-dimensional irreducible $\Fg$-module. It is well known that $V$ is both  a highest weight module and a lowest weight module. Hence, if $\dim V>1$, then for each weight vector $v$ of $V$, there exist weight vectors (or $0$) $v_\a,v_\a',\a\in\Phi^+$ such that   
\beqn
w=\sum_{\a\in \Phi^+}(e_\a v_\a+f_\a v_\a').
\eeqn

\medskip

We have the following proposition.

\begin{prop}\label{P3.1}
	Let $\Fg$ be a finite-dimensional semisimple complex Lie algebra and $V, W$ be $\Fg$-modules with $V$ finite-dimensional. Assume that a bilinear map $\d:\Fg\times V\longrightarrow W$ satisfying the following conditions
	\begin{itemize}
		\item[(1)] For each $x\in\Fg$, $\d(x,\cdot):V\longrightarrow W$ is a $\Fg$-module homomorphism;
		\item[(2)] For each $v\in V$, $\d(\cdot,v):\Fg\longrightarrow W$ is a derivation from $\Fg$ to $W$,
	\end{itemize}
	then $\d=0$.
\end{prop}
\begin{proof}
	We first assume that $V$ is irreducible with highest weight vector $v$ of highest weight $\La$ and lowest weight vector $v'$ of lowest weight $\La'$. Let $W'=\d(\Fg,V)$ be the image of $\d$, then
	\beqn
	W'=\sum_{\a\in\Phi^+}\(\d(e_\a,V)+\d(h_\a,V)+\d(f_\a,V)\).
	\eeqn
	It is obviously that $\dim W'<\infty$ and
	\beqn
	f_\a W'\cap W'_\La=e_\a W'\cap W'_{\La'}=\{0\},~\forall \a\in\Phi^+.
	\eeqn
	Since $\d(\cdot,v)$ and $\d(\cdot,v')$ are both derivations from $\Fg$ to $W'$, by Whitehead's lemma, there exist $w,w'\in W'$ such that
	\beqn
	\d(x,v)=xw,~\d(x,v')=xw', \forall x\in\Fg.
	\eeqn
	In particular, 
	\begin{align*}
		& \d(e_\a,v')=e_\a w'\in e_\a W'\cap W'_{\La'}=\{0\},\\
		& \d(f_\a,v)=f_\a w\in f_\a W'\cap W'_{\La}=\{0\},
	\end{align*}
	for any $\a\in\Phi^+$.
	It follows that
	\beqn
	\d(e_\a,V)=\d(f_\a,V)=0,~\forall \a\in\Phi^+,
	\eeqn
	since $\d(e_\a,\cdot)$ and $\d(f_\a,\cdot)$ are both $\Fg$-homomorphisms and $V$ is irreducible. 
	Since $\Fg$ is generated by the elements $e_{\a},f_{\a}$'s and $\d(\cdot,v)$ is a derivation for each $v\in V$, we obtain that 
	\beqn
	\d(\Fg,V)=0.
	\eeqn
	That is $\d=0$.
	
	\medskip
	
	Now, consider the general case. Let $V=V_1\oplus V_2\oplus \cdots\oplus V_m$ with each $V_i$ irreducible.
	For each $i$, the bilinear map $\d|_{\Fg\times V_i}$ satisfies the condition (1) and (2). Hence, $\d|_{\Fg\times V_i}=0$, i.e. $\d(\Fg,V_i)=0$ since $V_i$ is irreducible. We obtain
	\beqn
	\d(\Fg,V)=\d(\Fg,V_1\oplus V_2\oplus \cdots\oplus V_m)=\d(\Fg,V_1)+\d(\Fg,V_2)+\cdots+\d(\Fg,V_m)=0.
	\eeqn
	Equivalently, $\d=0$.
\end{proof}

Now we can prove our main result.

\begin{thm}\label{T3.1}
	Let $\Fg$ be a finite-dimensional semisimple complex Lie algebra and $V$ be a finite-dimensional $\Fg$-module. 
	Then all  the symmetric biderivations from $\Fg$ to $V$ are all trival.
\end{thm}

\begin{proof}
	Assume that $\d$ is a symmetric biderivations from $\Fg$ to $V$. Let $\a\in\Phi^+$. Note that the bilinear map $\d_{S_\a\times S_\a}$ gives a symmetric biderivations from $S_\a$ to $V$. Then, by Corollary \ref{C2.4}, $\d|_{S_\a\times S_\a}=0$, i.e.,
	\beq
	\d(x,y)=0,~\forall x,y\in S_\a.
	\eeq
    Let $x\in S_\a$. For any $y\in S_\a$ and $z\in\Fg$, one has
    \beqn
    \d(x,[y,z])=y\d(x,z)-z\d(x,y)=y\d(x,z).
    \eeqn
    It follows that 
    \beqn
    \d(x,\cdot):\Fg\longrightarrow V,~z\longmapsto \d(x,z)
    \eeqn
    is an $S_\a$-module homomorphism. Hence, the bilinear map $\d|_{S_\a\times\Fg}$ satisfies the conditions (1) and (2) in above proposition. It follows that $\d|_{S_\a\times\Fg}=0$. In particular,
    \beqn
    \d(e_\a,\Fg)=\d(f_\a,\Fg)=\d(h_\a,\Fg)=0.
    \eeqn
    Since $\Fg$ is spanned by the elements $e_\a,f_\a,h_\a~\a\in\Phi^+$, we have $\d(\Fg,\Fg)=0$, i.e. $\d=0$.
\end{proof}

We can have immediately the following corollary, which was also obtained in \cite{T18(1)}.
\begin{cor}
	Let $\Fg$ be a finite-dimensional complex semisimple Lie algebra. Then all the symmetric biderivations $\Fg$  are trival. 
\end{cor}

Moreover, we have 

\begin{cor}\label{C3.4}
	Let $\Fg$ be a finite-dimensional complex Lie algebra containing a semisimple subalgebra $\Fg_0$, $V$ be a finite-dimensional $\Fg$-module. 
	 If $\d$ is a symmetric biderivation from $\Fg$ to $V$, then $\d(\Fg_0,\Fg)=0$. 
\end{cor}
\begin{proof}
	 By Theorem \ref{T3.1}, $\d|_{\Fg_0\times\Fg_0}=0$. Note that $\Fg$ can be viewed as a $\Fg_0$-module via the adjoint action. Let $x\in\Fg_0$. For any $y\in\Fg_0,~z\in\Fg$, one has
	\beqn
	\d(x,[y,z])=y\d(x,z)-z\d(x,y)=y\d(x,z),
	\eeqn
	which implies that $\d(x,\cdot):\Fg\longrightarrow \Fg$ is a $\Fg_0$-module homomorphism. Hence, the bilinear map
	\beqn
	\d|_{\Fg_0\times \Fg}:\Fg_0\times\Fg\longrightarrow V
	\eeqn
	satisfies the conditions (1) and (2) Proposition \ref{P3.1}. It follows that $\d(\Fg_0,\Fg)=0$.
\end{proof}


\section{Applications}

In this section, we give some examples to show how to use Theorem \ref{T3.1}  to determine all the symmetric biderivations of some Lie algebras.

\subsection{Preliminaries}

Let's do some preparations for the further discussion.
The following lemma can be easily checked by a direct computation. 

\begin{lem}\label{L4.1}
	Let  $\d: \Fg\times\Fg\longrightarrow V$ be a symmetric biderivation. Then we have
	\beqn
	\d(x,[y,z])+\d(y,[z,x])+\d(z,[x,y])=0,~\forall x,y,z\in\Fg.
	\eeqn
\end{lem}

The following proposition can be deduced from the above lemma.
\begin{prop}\label{P4.2}
Let $\Fg$ be a complex Lie algbera and $V$ be a $\Fg$-module, $c$ be a central element in $\Fg$. Assume that $\d$ is a symmetric biderivation from $\Fg$ to $V$. Then $\d(\Fg^{(1)}, c)=0$.
\end{prop}

\begin{proof}
 For any $x,y\in\Fg$, by Lemma \ref{L4.1}, one has
\beqn
\d(c,[x,y])=\d(c,[x,y])+0+0=\d(c,[x,y])+\d(x,[y,c])+\d(y,[c,x])=0.
\eeqn
The proof is completed.
\end{proof}

We call a $\Fg$-module faithful if $Z_V(\Fg)=0$.
The following proposition is easy but useful.
\begin{prop}\label{P4.3}
	Let $\Fg$ be a finite-dimensional complex semisimple Lie algebra, and $W,V,U$  be  $\Fg$-modules with $W,V$ finite-dimensional and $V$ faithful. If a bilinear map $\d:W\times V\longrightarrow U$ satisfies the following conditions
	\begin{itemize}
		\item[(1)] For each $w\in W$, $\d(w,\cdot):V\longrightarrow U$ is a $\Fg$-module homomorphism;
		\item[(2)] For each $v\in V$, $\d(\cdot,v):V\longrightarrow U$ is a $\Fg$-module homomorphism,
	\end{itemize}
	then $\d=0$.
\end{prop}

\begin{proof}
	We first assume that $V$ is irreducible. Then $\dim V>1$ since $V$ is a faithful $\Fg$-module. Let $w\in W$ and $v\in V$ be two weight vectors. If $\wt(w)\neq \wt(v)$, then $\d(w,v)=0$.  Assume that $\wt(w)=\wt (v)$. We can find weight vectors (or $0$) $v_\a,v_\a',\a\in\Phi^+$ such that   
	\beqn
	v=\sum_{\a\in \Phi^+}(e_\a v_\a+f_\a v_\a').
	\eeqn
	We have
	\beqn
	\d(w,v)=\sum_{\a\in \Phi^+}(\d(w,e_\a v_\a)+\d(w,f_\a v_\a'))=\sum_{\a\in \Phi^+}(e_\a\d(w, v_\a)+f_\a\d(w, v_\a'))=0.
	\eeqn
	It follows that $\d=0$.
	
	\medskip
	
	Assume that $V$ is reducible. Then by Weyl's theorem, $V$ can be written as a direct sum of irreducible $\Fg$-submodules
	\beqn
	V=V_1\oplus V_2\oplus \cdots\oplus V_m.
	\eeqn
	Since $V$ is faithful, $V_i$ is faithful for each $1\le i\le m$. By the above discussion, 
	\beqn
	\d(W,V_i)=0,~i=1,2,\cdots, m.
	\eeqn
	Hence,
	\beqn
	\d(W,V)=\d(W,V_1\oplus V_2\oplus \cdots\oplus V_m)=\d(W,V_1)+\d(W,V_2)+\cdots+\d(W,V_m)=0.
	\eeqn
	That is $\d=0$.
\end{proof}

We can prove the following two theorems.

\begin{thm}\label{T4.4}
	Let $\Fg$ be a finite-dimensional complex  Lie algbera containing a semisimple Lie algbera $\Fg_0$, $V$ a finite-dimensional $\Fg$-module. If  $\Fg$ is faithful as a $\Fg_0$-module,
	 then all the symmetric biderivations from $\Fg$ to $V$ are trival. 
\end{thm}

\begin{proof}
	By Corollary \ref{C3.4}, $\d(\Fg_0,\Fg)=0$. Let $x\in\Fg$. For any $y\in\Fg_0$ and $z\in\Fg$, one has
	\begin{align*}
	\d(x,[y,z])&=y\d(x,z)-z\d(x,y)=y\d(x,z),\\
	\d([y,z],x)&=y\d(z,x)-z\d(y,x)=y\d(x,z).
	\end{align*}
	Hence, $\d(\cdot,x)$ and $\d(x,\cdot)$ are both $\Fg_0$-module homomorphisms from $\Fg$ to $V$. Since, $\Fg$ is faithful as a $\Fg_0$-module, we have
	$\d=0$ by Proposition \ref{P4.3}.
\end{proof}


\begin{thm}\label{T4.5}
	Let $\d$ be a symmetric biderivation from $\Fg$ to $V$.
	Let $\Fg$ be a finite-dimensional complex perfect Lie algbera containing a semisimple Lie algbera $\Fg_0$, $V$ a finite-dimensional $\Fg$-module. If the center of  $\Fg$ equals  
	\beqn
	Z_\Fg(\Fg_0)=\{ x\in \Fg\mid [y,x]=0~\text{for all~} y\in \Fg\},
	\eeqn
	then all the symmetric biderivations from $\Fg$ to $V$ are trival. 
\end{thm}

\begin{proof}
	Let $\d$ be a symmetric biderivation from $\Fg$ to $V$.
	As a $\Fg_0$-module, $\Fg$ can be decomposed as
	\beqn
	\Fg=W \oplus Z_\Fg(\Fg_0),
	\eeqn
	where $W$ is a faithful $\Fg_0$-submodule of $\Fg$. By Corollary \ref{C3.4}, $\d(\Fg_0,\Fg)=0$. 
	
	\medskip
	
	Let $x\in\Fg$ and $w\in W$. For any $y\in\Fg_0$, one has
	\begin{align*}
		\d(x,[y,w])&=y\d(x,w)-w\d(x,y)=y\d(x,z),\\
		\d([y,x],w)&=y\d(x,w)-x\d(y,w)=x\d(w,v).
	\end{align*}
	Hence, $\d(x,\cdot):W\longrightarrow V$ and $\d(\cdot,w):\Fg\longrightarrow V$ are both $\Fg_0$-module homomorphisms. Since, $W$ is a faithful  $\Fg_0$-module, by Proposition 4.3, we have
	\beq
	\d(x,w)=0,~\forall x\in \Fg,w\in W.
	\eeq
	
	Since $\Fg$ is perfect, by Proposition \ref{P4.2}, 
	\beqn
	\d(\Fg,Z_\Fg(\Fg_0) )=\d(\Fg^{(1)},Z_\Fg(\Fg_0))=0.
	\eeqn
	The proof is completed.
\end{proof}

\subsection{Reductive Lie algebras}

A reductive Lie algebra is a finite-dimensional central extension  of certain semisimple Lie algebras. In the following, we just consider the simplest case. 
We remark here that the   discussion can be easily applied to determine  the symmetric biderivations  of a reductive Lie algebra. We leave it to the interested readers.

\medskip

Let $\Fg$ be the $1$-dimensional central extension of a semisimple Lie algebra $\Fg_{0}$, i.e. $\Fg=\Fg_{0}\oplus \CC c$ with $[c,\Fg]=0$.   We will determine the symmetric biderivations of  $\Fg$.

\medskip

Let $\d$ be a symmetric biderivation of $\Fg$. Then by Theorem \ref{T3.1} and Proposition \ref{P4.2}, we have 
\beqn
\d(\Fg_0,\Fg_0)=0,~\d(\Fg^{(1)},c)=\d(\Fg_0,c)=0.
\eeqn
Since $\d(\cdot,c)$ is a derivation of $\Fg$, there exists some $\la\in\CC$ such that 
\beqn
\d(c,c)=\la c.
\eeqn
Conversely, for each $\la\in \CC$, the formula
\beq
\d(x+\mu_1 c,y+\mu_2 c)=\mu_1\mu_2\la c,~\forall x,y\in\Fg_0,\mu_1,\mu_2\in\CC
\eeq
indeed determines a symmetric biderivation of $\Fg$. 
To summarize,  we obtain
\begin{thm}
Let $\Fg=\Fg_0\oplus \CC c$ be a reductive Lie algbera, where $\Fg_0$ is semisimple. Then each symmetric biderivation $\d$ of $\Fg$ is of the form
\beqn
\d(x+\mu_1c,y+\mu_1c)=\la\mu_1\mu_2 c,~\forall x,y\in \Fg_0,\mu_1,\mu_2\in\CC,
\eeqn
for  a   complex number $\la$. 
\end{thm}

\subsection{The  Takiff algebras}

Let $\Fg$ be a finite-dimensional complex semisimple Lie algebra, the Takiff algebra \cite{Lau24}  associated $\Fg$ is, by definition,  
$\tilde{\Fg}=\Fg\otimes \CC[t]/(t^2)$.  Let $V=\Fg\otimes \o{t}$. Then $V$ is an ideal of $\tilde{\Fg}$. As a $\Fg$-module, $V$ is isomorphic to the adjoint representation. Note that $\tilde{\Fg}=\Fg\ltimes V$ with bracket 
\beqn
[x+v,y+w]=[x,y]+xw-yv,~\forall x,y\in \Fg, v,w\in V.
\eeqn
Here,
\beqn
x(y\otimes \o{t})=[xy]\otimes\o{t},~\forall y\in\Fg.
\eeqn
We can obtain the following theorem.

\begin{thm}
All the symmetric biderivations of $\tilde{\Fg}$ are trival.
\end{thm}
\begin{proof}
Let $\d$ be a symmetric biderivation of $\tilde{\Fg}$. 
Note that $\tilde{\Fg}$ is isomorphic to $V\oplus V$ as a $\Fg$-module, which is faithful. Hence, by Theorem \ref{T4.4}, $\d=0$.
\end{proof}
\subsection{The symplectic oscillator algebras}

Following \cite{Liu21}, we  recall the symplectic oscillator algebra $\Fg_n$. Recall that the symplectic complex Lie algebra $\Fsp_n(\CC)$ can be realized the matrix Lie algbera consisting all the $2n\times 2n$ complex matrices with block form
\beqn
\begin{pmatrix}
A & B\\
C & -A^\T
\end{pmatrix},
\eeqn
where $A,B,C$ are $n\times n$ matrices with $B=B^\T$ and $C=C^\T$.
It is well known that $\Fsp_{2n}$ has the natural representation on $\CC^{2n}$ by the left matrix multiplication. 
Note that $\CC^{2n}$ is an irreducible $\Fsp_{2n}(\CC)$-module via the natural action.

\medskip

The Heisenberg Lie algbera $H_n=\CC^{2n}\oplus \CC z$ is the Lie algbera with bracket given by
\beqn
[e_i,e_{n+i}]=z,~[z,H_n]=0.
\eeqn 
The symplectic oscillator algebra  $\Fg_n$ is the semidirect product Lie algbera 
\beqn
\Fg_n=\Fsp_{2n}(\CC)\ltimes H_n
\eeqn 
with bracket 
\beqn
[x,v]=xv,~[x,z]=0
\eeqn
for all $x\in\Fsp_{2n}(\CC),~v\in\CC^{2n}$.  It is easy to see that $\Fg_n$ is perfect and 
\beqn
Z_{\Fg_n}({\Fsp_{2n}(\CC)})=\CC z.
\eeqn
By Theorem \ref{T4.5}, we obtain the following theorem immediately 

\begin{thm}
All the symmetric biderivations of $\Fg_n$ are trival.
\end{thm}

\subsection{The Schr$\ddot{\boldsymbol{\mathrm{o}}}$dinger algebras} Let us recall the definition of the Schr$\ddot{\mathrm{o}}$dinger algebra $\S_n$ in $(n+1)$-dimensional space-time following \cite{LLW}. The Schr$\ddot{\mathrm{o}}$dinger algebra $\S_n$ is a finite-dimensional, non-semisimple and perfect complex Lie algbera, and it is the semidirect product Lie algebra
\beqn
\S_n=(\Fsl_2(\CC)\oplus\Fsp_{2n}(\CC))\ltimes H_n,
\eeqn
where $H_n=\CC^{2n}\oplus\CC z$ is the Heisenberg Lie algbera. The element $z$ is central in $\S_n$. Let $\Fg_0=\Fsl_2(\CC)\oplus\Fsp_{2n}(\CC)$. Then $\Fg_0$ is a semisimple subalgebra of $\S_n$ and $Z_{\S_n}(\Fg_0)=\CC z$ is the center of $\S_n$. It is obviously that $\S_n$ is perfect. By Theorem \ref{T4.5}, we obtain  

\begin{thm}
	All the symmetric biderivations of $\S_n$ are trival.
\end{thm}


\subsection{Lie superalgebras}

Let $\Fg=\Fg_{\o{0}}\oplus\Fg_{\o{1}}$ be a (complex) Lie superalgebra. A homogeneous  bilinear map $\d:\Fg\times\Fg\longrightarrow \Fg$ is said to be a super-biderivation of $\Fg$ if it satisfies 
\begin{align}
&\d([x,y],z)=(-1)^{|\d||x|}[x,\d(y,z)]-(-1)^{|y|(|\d|+|x|)}[y,\d(x,z)],\\
&\d(x,[y,z])=(-1)^{(|\d|+|x|)|y|}[y,\d(x,z)]-(-1)^{|z|(|\d|+|x|+|y|)}[z,\d(x,y)]
\end{align}
for all homogeneous elements $x,y,z\in\Fg$. See \cite{D23,Y18} for details.
Note that $\d|_{\Fg_{\o{0}}\times\Fg_{\o{0}}}$ is a biderivation from $\Fg_{\o{0}}$ to $\Fg$ (viewed as a $\Fg_{\o{0}}$-module via the adjoint action).

\medskip

A super-biderivation $\d$ is said to be super-symmetric, if
\beqn
\d(x,y)=(-1)^{|x||y|}\d(y,x),~\forall x,y\in\Fg.
\eeqn
We have the following theorem.

\begin{thm}
Let $\Fg=\Fg_{\o{0}}\oplus\Fg_{\o{1}}$ be a perfect  finite-dimensional complex Lie superalgebra with $\Fg_{\o{0}}$ reductive. Let $\Fg_0=[\Fg_{\o{0}}, \Fg_{\o{0}}]$. Assume that the dimension of 
\beqn
Z_{\Fg_{\o{1}}}(\Fg_0)=\{v\in\Fg_{\o{1}}\mid [x,v]=0~\text{for all}~x\in\Fg_0\}
\eeqn
is at most one.
Then all the super-symmetric biderivations of $\Fg$ are trival.
\end{thm}

\begin{proof}
 Note that $\Fg_0$ is semisimple. Let $\Phi$ be the root system of $\Fg_0$ and $\Phi^+$ the set of positive roots.  Both $\Fg$ and $\Fg_{\o{1}}$ can be viewed as $\Fg_0$-modules via the adjoint action.
Assume that $\d$ is a super-symmetric biderivation of $\Fg$. By Theorem \ref{T3.1}, we have
\beq
\d(x,y)=0, ~\forall x,y\in\Fg_0.
\eeq
Since $\Fg_0=[\Fg_{\o{0}}, \Fg_{\o{0}}]$, by Proposition \ref{P4.2}, one has 
\beq
\d(\Fg_0,c)=\d([\Fg_{\o{0}}, \Fg_{\o{0}}],c)=0
\eeq
for any central element $c$ of $\Fg_{\o{0}}$.

\medskip

Let $x\in\Fg_0$. For any $y\in\Fg_{0}$ and $v\in\Fg_{\o{1}}$, one has
\beqn
\d(x,[y,v])=(-1)^{(|\d|+|x|)|y|}[y,\d(x,v)]-(-1)^{|v|(|\d|+|x|+|y|)}[v,\d(x,y)]=[y,\d(x,v)].
\eeqn
It implies that $ \d(x,\cdot)|_{\Fg_{\o{1}}}$ is a $\Fg_0$-module homomorphism from $\Fg_{\o{1}}$ to $\Fg$. Let $v\in\Fg_{\o{1}}$. Then for any $x,y\in\Fg_0$, one has
\beqn
\d([x,y],v)=(-1)^{|\d||x|}[x,\d(y,v)]-(-1)^{|y|(|\d|+|x|)}[y,\d(x,v)]=[x,\d(y,v)]-[y,\d(x,v)].
\eeqn
It implies that $\d(\cdot,v)|_{\Fg_0}$ is a derivation from $\Fg_0$ to $\Fg$. By Proposition \ref{P4.3}, $\d|_{\Fg_0\times \Fg_{\o{1}}}=0$, i.e.,
 \beq
 \d(x,v)=0,~\forall x\in\Fg_0, v\in\Fg_{\o{1}}.
 \eeq
 
 Let $v\in\Fg_{\o{1}}$. For any $x\in\Fg_0$ and $w\in\Fg_{\o{1}}$, we have
 \begin{align*}
  \d([x,w],v)&=(-1)^{|\d||x|}[x,\d(w,v)]-(-1)^{|v|(|\d|+|x|)}[w,\d(x,v)]=[x,\d(w,v)],\\
  \d(v,[x,w])&=(-1)^{|v||[x,w]|}\d([x,w],v)=-[x,\d(w,v)]=-[x,(-1)^{|w||v|}\d(v,w)]\\
                &=[x,\d(v,w)].
 \end{align*}
Hence, $\d(\cdot,v)|_{\Fg_{\o{1}}}$ and $\d( v,\cdot)|_{\Fg_{\o{1}}}$ are both $\Fg_0$-module homomorphism from $\Fg_{\o{1}}$ to $\Fg$. 

\medskip

By Weyl's theorem, $\Fg_{\o{1}}$ can be written as a direct sum of irreducible $\Fg_0$-submodules, say
\beqn
\Fg_{\o{1}}=V_1\oplus V_2\oplus\cdots\oplus V_m.
\eeqn
We claim that there exist at most one $i$ such that $\dim V_i=1$. Otherwise, there exist  $i\neq j$ such that 
$\Fg_0V_i=\Fg_0V_j=0$. Then
\beqn
V_i\oplus V_j\subseteq Z_{\Fg_{\o{1}}}(\Fg_{0}),
\eeqn
which forces $\dim Z_{\Fg_{\o{1}}}(\Fg_{0})\geq 2$, a contradiction. By Proposition \ref{P4.3}, 
\beqn
\d(V_i,V_j)=0
\eeqn
if $\dim V_i>1$ or $\dim V_j>1$.
Assume that $v,w\in V_i$ with $\dim V_i=1$. Then $v=aw$ for some $a\in\CC$. Since 
\beqn
\d(w,w)=(-1)^{|w||w|}\d(w,w)=-\d(w,w), 
\eeqn
we have $\d(w,w)=0.$
Hence,
\beqn
\d(v,w)=a\d(w,w)=0,
\eeqn
We obtain that
\beqn
\d(v,w)=0,~\forall w,v\in\Fg_{\o{1}}.
\eeqn

\medskip

Let $c$ be a central element of $\Fg_{\o{0}}$. Then $c\in[\Fg_{\o{1}},\Fg_{\o{1}}]$. Hence, we must have
\beq
\d(c,\Fg_{\o{1}})=0.
\eeq
It also implies that
\beq
\d(c_1,c_2)=0
\eeq
for any central elements $c_1,c_2$. We complete the proof.
\end{proof}

Since any classical Lie superalgebra satisfies the condition of above theorem. We can obtain the following corollary, 
which has been obtained by concrete, but a little complicated, computations in \cite{Y18}. 
\begin{cor}
All the super-symmetric biderivations of a classical simple Lie superalgebra are trival.
\end{cor}



\end{document}